\documentclass{amsart}

\title[Basins of Attraction of Automorphisms in $\CC^3$]{Basins of Attraction of Automorphisms  in $\CC^3$}
\author{Liz Vivas and Berit Stensones}
\date{October 5th, 2011}

\usepackage{amscd}

\newtheorem*{theorem}{Theorem}
\newtheorem{lemma}{Lemma}

\theoremstyle{definition}

\theoremstyle{remark}
\newtheorem{remark}{Remark}

\newcommand{\CC}{\mathbb{C}}
\newcommand{\PP}{\mathbb{P}}

\def\Id{{\mathrm{Id}}}

\begin{document}

\bibliographystyle{plain}

\begin{abstract}
In this paper we shall give examples of maps and automorphisms with regions of attraction that are not simply connected.
\end{abstract}

\maketitle

\section{introduction}

The objects that we shall study are holomorphic maps $F: \CC^k \to \CC^k$ with a fixed point. For simplicity we may assume that $F(O) = O$. An open set $\Omega = \{z \in \CC^k, F^n(z) \to O,\textrm{ when }n \to \infty\}$ is called a \textit{basin of attraction}.

If $k=1$, it is known that if $F$ does have a basin of attraction $\Omega$, then $\Omega$ is a disjoint union of simply connected regions in $\CC$ \cite{Mi}.

We shall show that this is no longer the case when $k\geq 2$.

It is easy to show that if $O$ is an interior point for $\Omega$, then $\Omega$ will have to be one simply connected region so we have to look for our examples among the maps where $O$ is in the boundary of $\Omega$.

A class of maps that have been studied extensively are the ones where $DF(O) = \textrm{Id}$, these are said to be \textit{tangent to the identity}.

The main theorem of this paper is the following:

\begin{theorem} For any $k \geq 3$, there exist $F$ automorphism of $\CC^k$ tangent to the identity at $O$ and its basin $\Omega$ is biholomorphic to $(\CC^*)^{k-2}\times \CC^2$.
\end{theorem}

\section{Definitions and motivations}

We shall use a very famous example from complex dynamics in one variable as an inspiration for our constructions. Let $f(z) = z + az^2 = z (1+az)$ where $a$ is real and non-zero, then $f(0) = 0$ and $f'(0) = 1$.

The basins of attractions $C_{a}$ for these maps are simply connected, bounded and $0 \in \partial C_a$. The map $z \to z(1+\frac{1}{2}z)$ behaves in a similar way. 
In $\CC^2$ we can define the map $F(z,w) = \left(z+\frac{1}{2}z^2w, w + \frac{1}{2}zw^2\right) = (z,w)\left(1+\frac{1}{2}zw\right)$. If we look at how this map acts on $zw$, we observe that $zw \to zw(1+\frac{1}{2}zw)^2$. From this we see that if $(z_n,w_n) = F^n(z,w)$, then $\prod_{j=0}^n\left(1+\frac{1}{2}z_nw_n\right)^2 \to O$ as $n \to \infty$, if $zw \in \tilde{C}$, where $\tilde{C}$ is a region very similar to the domains $C_a$. If the product $zw \notin \tilde{C}$, then $z_nw_n$ does not go to $0$. Hence the basin of attraction for $F$ is the region
$\{(z,w)\in\CC^2, zw \in \tilde{C}\}$. Since the map is equal to the identity on the axes, if follows that the map $(z,w) \to (z,zw)$ is one to one in $\Omega$, so $\Omega$ is biholomorphic to $\CC^* \times \tilde{C}$.

The map $F$ is not an automorphism and we have not been able to find an automorphism of $\CC^2$ with a similar behavior. 

The situation changes if we allow $k$ to be larger than or equal to $3$. We will give the construction in the case $k=3$ and indicate how it can be done for larger $k$'s.

\begin{lemma}
For any given $a,b,c$ different than $0$, there exists $F = (F_1,F_2,F_3): \CC^3 \to \CC^3$ automorphism tangent to the identity of the form:
\begin{align}
F_1(z,t,w) &= z[1 - a\zeta + O(\zeta^2,\zeta w)] \nonumber \\
F_2(z,t,w) &= t[1 - b\zeta + O(\zeta^2,\zeta w)]   \label{formF}\\
F_3(z,t,w) &= w[1 - c\zeta + O(\zeta^2,\zeta w)] + O(\zeta^3) \nonumber;
\end{align}
where $\zeta = zt$. 
\end{lemma}

\begin{proof}
We construct $F$ as a finite composition of shears and overshears. We use the following shears and overshears:
\begin{align*}
\phi_1(z,t,w) &= (z,t,w-zt),\\
\phi_2(z,t,w) &= (ze^{aw},te^{bw},w),\\
\phi_3(z,t,w) &= (z,t,we^{-(a+b+c)zt}+(a+b)z^2t^2).
\end{align*}
Then we define $F$ as follows:
\begin{align*}
F = \phi_3 \circ \phi_2^{-1} \circ \phi_1^{-1}\circ \phi_2 \circ \phi_1. 
\end{align*}
An easy computation shows that $F$ has the form  \eqref{formF}, as above.
\end{proof}

We choose and fix from now on $a,b,c$ positive and real, $a=b$ and $c > 2a$.  Note that with this choice $tF_1(z,t,w) = zF_2(z,t,w)$.
Also, notice that $F(z,0,w) = (z,0,w)$ and $F(0,t,w) = (0,t,w)$.

It remains to prove that $F$ does have a basin of attraction that is biholomorphic to $\CC^*\times\CC^2$.

\section{Proof of the theorem}

To prove the main theorem we shall use the ideas developed by Hakim \cite{hak}. Given a map $F:\CC^k\to\CC^k$ that fixes the origin and it is tangent to the identity at $O$, then we can write:
$$
F(z) = z + P_r(z) + O(|z|^{r+1}).
$$
where $P_r\neq 0$ the $r$-th homogeneous term in the homogeneous expansion of $F$.

We say that $v\neq O$ is a \textit{characteristic direction} if $P_r(v) = \lambda v$. If $\lambda \neq 0$ then we say $v$ is \textit{non-degenerate} and we say $v$ is \textit{degenerate} if $\lambda = 0$. We identify $P_k$ with its induced map in $\PP^{k-1}$ as well as $v$ with its projection $[v]$ in $\PP^{k-1}$.

We say that an orbit $\{F^n(p)\}$ converges to the origin tangentially to a direction $v$ in $\CC^n$ if $F^n(p) \to O$ and $[F^n(p)] \to [v]$ in $\PP^{k-1}$, where $[\cdot]:\CC^k \backslash\{O\} \to \PP^{k-1}$ denotes the canonical projection. Characteristic directions are relevant in the study of orbits going to the origin due to the following fact \cite{hak}: If there is an orbit of $F$ converging to the origin tangentially to a direction $v \in \CC^k$ then $v$ is a characteristic direction of $F$.

To any non-degenerate characteristic direction we can associate a set of eigenvalues $\alpha_1,\ldots,\alpha_{k-1} \in \CC$ of the linear operator $D(P_r)([v])-\Id: T_{[v]}\PP^{n-1}\to\PP^{n-1}$. These are called the \textit{directors} of $[v]$. Hakim (see also Weickert \cite{we}) proved the following: if $F$ is a biholomorphism, $[v]$ a non-degenerate characteristic direction and the directors of $[v]$ have all positive real part, then $\Omega_{[v]}$ is biholomorphic to $\CC^k$. 

For $F$ as above, in lemma 1, we have that the characteristic directions are $(z,0,w)$ and $(0,t,w)$. All the directions are degenerate.

Define the map $\pi: \CC^3 \to \CC^2$ as follows $\pi(z,t,w) = (zt,w)$. Then the following diagram commutes:

\[ 
\begin{CD} 
\CC^3 @>F>> \CC^3\\ 
@V\pi VV  @V\pi VV   \\ 
\CC^2 @>G>> \CC^2
\end{CD} \] 
where $G$ is of the form:
\begin{align}
G(\zeta,w) = (\zeta - 2a\zeta^2 + O(\zeta^3,\zeta^2w), w - c\zeta w + O(\zeta^3, \zeta^2w, \zeta w^3)),
\end{align}
for $\zeta = zw$.

The basin of attraction for $F$ and $G$ are $\Omega=  \{ p \in \CC^3 \backslash O, F^n(p) \to O\}$ and $\tilde\Omega= \{ q \in \CC^2 \backslash O, G^n(q) \to O\}$ respectively.

We see the following for $\Omega$: if $(z_o,t_o,w_o) \in \Omega$, then $z_ow_o \neq 0$. Also, for any other $(z,t_o,w)$ such that $zw = z_ow_o$ we will have that $(z,t_o,w)$ is also in $\Omega$.

The map $G$ has the non-degenerate characteristic direction $(1,0)$. We can compute its director and we obtain $\frac{c-2a}{2a}$. With our choice of $a$ and $c$ we obtain that the director is positive, therefore $\tilde{\Omega}$ is biholomorphic to $\CC^2$.

\begin{lemma} $\pi(\Omega) = \tilde{\Omega}$. 
\end{lemma}

\begin{proof}

Let $p \in \Omega$, then for $q = \pi(p)$ we have $G^n(q) = \pi(F^n(p))\to \pi(O) = O$. Then $\pi(\Omega)\subset \tilde{\Omega}$.
Now let $q = (x,y) \in \tilde{\Omega}$, we claim that $p=(\sqrt{x},\sqrt{x},y)$ is in $\Omega$ i.e. $F^n(\sqrt{x},\sqrt{x},y) \to O$. 
Note that $x \neq 0$, since every point of the form $(0,y)$ is fixed by $G$ and therefore does not converge to the origin. Let $(z_n,t_n,w_n) = F^n(\sqrt{x},\sqrt{x},y)$, then $z_n = t_n$. Since $G^n(q) = \pi(F^n(p)) \to O$, we have $z_nt_n \to 0$ and $w_n \to 0$. Therefore $z_n = t_n \to 0$ and $w_n \to 0$. We conclude $p \in \Omega$.

\end{proof}
%

Now we are ready to prove that $\Omega$ is biholomorphic to $\CC^\star \times \CC^2$.

\begin{proof}[Proof of Theorem]
Hakim's theorem say that $\tilde\Omega$ is biholomorphic to $\CC^2$. Let us call $\phi: \tilde\Omega \to \CC^2$ a chosen biholomorphism.
Then define the map:
\begin{align*}
\Psi &: \Omega \to \CC^3\\
\Psi(z,t,w) &= (z,\phi(zt,w)).
\end{align*}
Clearly this map is injective, since $\phi$ is injective and $\Omega$ does not intersect the set $\{zt=0\}$. Therefore $\Psi$ is a biholomorphism between $\Omega$ and its image. It is not hard to see that $\Psi(\Omega) = \CC^* \times \CC^2$.
\end{proof}

\begin{remark}
For the construction of a basin $\Omega$ of an automorphism $F$ in $\CC^{k+1}$ such that is biholomorphic to $(\CC^\star)^{k-1}\times\CC^2$ for $k \geq 3$, we choose a map $F$ as follows $F(z,w) = (F_1,\ldots,F_{k+1})$ for $(z,w) = (z_1,\ldots,z_k,w)$:
\begin{align*}
F_i(z,w) &= z_i(1-a_i\zeta + O(\zeta^2,\zeta w))\textrm{ for }1\leq i \leq k\textrm{ and}\\
F_{k+1}(z,w) &= w-b\zeta w + O(\zeta^3,\zeta^2 w,\zeta w^3)
\end{align*}
where $\zeta = \prod_{i=1}^kz_i$. The same conclusion will follow if we use all $a_i$ equal, real and positive and $b>\sum a_i$. 
\end{remark}

\begin{remark}
In each one of these cases we have that the set of characteristic directions is $k$ dimensional. Following Abate and Tovena's terminology (see \cite{ab-tov}) we can also see that each of these maps is $0$-dicritical. That is, the set of singular directions is $0$-dimensional.
\end{remark}


\begin{thebibliography}{ab-br-to}

\addtolength{\leftmargin}{4in} 
\setlength{\itemindent}{-0.2in} 

\bibitem[Ab-To]{ab-tov}
	M. Abate and F. Tovena, \emph{Parabolic curves in ${\CC}^3$}, Abstr. Appl. Anal. \textbf{5}, (2003), 275--294.

\bibitem[Hak]{hak}
	M. Hakim, \emph{Analytic transformations of {$(\bold C\sp p,0)$} tangent to the identity}, Duke Math. J. \textbf{92} (1998), 403--428.


\bibitem[Mi]{Mi}
	J. Milnor, \emph{Dynamics in one complex variable}, third ed., Annals of Mathematics Studies, vol. 160, Princeton University Press, Princeton, NJ, 2006.

%

\bibitem[We]{we}
		B. Weickert, \emph{Attracting basins for automorphisms of {$\CC^2$}}, Invent. Math. \textbf{132} (1998), 581--605.


\end{thebibliography}
\end{document}